%% file: main.tex
\documentclass[10pt]{article} 
\include{00Prae}
 
\title{Functional Tucker approximation using Chebyshev interpolation}
\usepackage{authblk}
\author[1]{Sergey Dolgov}
\author[2]{Daniel Kressner}
\author[3]{Christoph Strössner}
\affil[1]{Department of Mathematical Sciences, University of Bath, United Kingdom, \texttt{S.Dolgov@bath.ac.uk}}
\affil[2]{Institute   of   Mathematics, EPF  Lausanne,  Switzerland,  \texttt{daniel.kressner@epfl.ch}.}
\affil[3]{Institute   of   Mathematics, EPF  Lausanne,  Switzerland,  \texttt{christoph.stroessner@epfl.ch}.}
\date{March 3, 2021}

\begin{document}
\maketitle

\begin{abstract} 
This work is concerned with approximating a trivariate function defined on a tensor-product domain via function evaluations. Combining tensorized Chebyshev interpolation with a Tucker decomposition of low multilinear rank yields function approximations that can be computed and stored very efficiently.
The existing Chebfun3 algorithm [Hashemi and Trefethen, \emph{SIAM J. Sci. Comput.}, 39 (2017)]
uses a similar format but the construction of the approximation proceeds indirectly, via a so called slice-Tucker decomposition. As a consequence, Chebfun3 sometimes uses unnecessarily many function evaluation\revision{s} and does not fully benefit from the potential of the Tucker decomposition to reduce, sometimes dramatically, the computational cost. We propose a novel algorithm Chebfun3F that utilizes univariate fibers instead of bivariate slices to construct  the Tucker decomposition. Chebfun3F reduces the cost for the approximation \revision{in terms of the number of function evaluations} for nearly all functions considered, typically by 75\%, and sometimes by over 98\%.
\end{abstract}

\input{1Introduction}

\input{2Theory}

\input{3Chebfun3}

\input{4Chebfun3F}

\input{5Results}

\input{6Conclusions}

{\input{main.bbl}}

\end{document}

%% file: 00Prae.tex
\usepackage{fullpage}
\usepackage[T1]{fontenc} 
\usepackage[latin1]{inputenc}
\usepackage[english]{babel}
\usepackage{selinput}
\SelectInputMappings{adieresis={ä},germandbls={ß}}

\usepackage{amsmath} 
\usepackage{amssymb} 
\usepackage{amsthm} 
\usepackage{enumerate}
\newtheorem{theorem}{Theorem}
\newtheorem{lemma}{Lemma}

\usepackage{graphicx} 
\usepackage{epstopdf}
\usepackage{subfig}
\usepackage[framed,numbered]{matlab-prettifier}
\usepackage{algorithm} 
\usepackage[noend]{algpseudocode}
\renewcommand\algorithmicthen{}
\renewcommand\algorithmicdo{}
\graphicspath{{./images/}}

\usepackage{multirow}
\usepackage{fancyvrb}

\newcommand{\R}{\mathbb{R}}
\newcommand{\bra}[1]{\left( #1 \right) }
\newcommand{\set}[1]{\left\lbrace #1 \right\rbrace}
\newcommand{\norm}[1]{\lVert #1 \rVert}
\newcommand{\abs}[1]{\left\lvert #1 \right\rvert}
\newcommand{\floor}[1]{\left\lfloor #1 \right\rfloor}

\def\defeq{\mathrel{\mathop:}=} 

\renewcommand{\b}[1]{\mathbf{#1}}
\newcommand{\eps}{\varepsilon}
\newcommand{\diag}{\text{diag}}

\DeclareMathOperator{\argmin}{argmin}
\DeclareMathOperator{\argmax}{argmax}
\newcommand{\uargmin}[1]{\underset{#1}{\argmin}\;}
\newcommand{\uargmax}[1]{\underset{#1}{\argmax}\;}
\newcommand{\umin}[1]{\underset{#1}{\min}\;}
\newcommand{\umax}[1]{\underset{#1}{\max}\;}
\newcommand{\usup}[1]{\underset{#1}{\sup}\;}
\newcommand{\uinf}[1]{\underset{#1}{\inf}\;}

\newcommand{\revision}[1]{\textcolor{black}{#1}}
\newcommand{\sRevision}[1]{\textcolor{black}{#1}}

\newcommand{\T}{\mathcal{T}}

\newcommand{\tU}{\tilde{U}}
\newcommand{\tV}{\tilde{V}}
\newcommand{\tW}{\tilde{W}}
\newcommand{\tI}{\tilde{I}}
\newcommand{\tJ}{\tilde{J}}
\newcommand{\tK}{\tilde{K}}
\renewcommand{\abs}[1]{|#1|}

\usepackage{tikz}
\usetikzlibrary{shapes,arrows}

%% file: 1Introduction.tex
\section{Introduction}

This work is  concerned with the approximation of trivariate functions (that is, functions depending on three variables) defined on a tensor-product domain, for the purpose of performing numerical computations with these functions. 
Standard approximation techniques, such as interpolation on a regular grid, may require an impractical amount of function evaluations. 
Several techniques have been proposed to reduce the number of evaluations for multivariate functions by exploiting additional properties. 
For example, sparse grid interpolation~\cite{Bungartz04} exploits mixed regularity. Alternatively, functional low-rank (tensor) decompositions, such as the spectral tensor train decomposition~\cite{Bigoni16}, the continuous low-rank decomposition~\cite{Gorodetsky17, Gorodetsky19}, and the QTT decomposition~\cite{Khoromskij2011,Oseledets2009i}, have been proposed.
In this work, we focus on using the Tucker decomposition for third-order tensors, following work by Hashemi and Trefethen~\cite{Hashemi17}.

The original problem of finding separable decompositions of functions is intimately connected to low-rank decompositions of matrices and tensors~\cite[Chapter 7]{Hackbusch12}.
A trivariate function is called separable if it can be represented as a product of univariate functions: $f(x,y,z) = u(x) v(y) w(z)$. 
If such a decomposition is available, it is usually much more efficient to work with the factors $u,v,w$ instead of $f$ when, e.g., discretizing the function. In practice, most functions are usually not separable, but they often can be well approximated by a sum of separable functions.
Additional structure can be imposed on this sum, corresponding to different tensor formats. 
In this work, we consider the approximation of a function $f: [-1,1]^3 \to \R$ in the functional Tucker format~\cite{Tucker66}, as in~\cite{Hashemi17, Luu17, Rai19}, which takes the form
\begin{equation} \label{eq:tuckerapprox}
f(x,y,z) \approx \sum_{i=1}^{r_1} \sum_{j=1}^{r_2} \sum_{k=1}^{r_3} \sRevision{\mathcal{C}}_{ijk} u_i(x) v_j(y) w_k(z),
\end{equation}
with univariate functions $u_i, v_j, w_k \colon [-1,1] \to \mathbb{R}$ and the so called \emph{core tensor} $\mathcal{C} \in \R^{r_1\times r_2\times r_3}$.
The minimal $r = (r_1,r_2,r_3)$, for which~\eqref{eq:tuckerapprox} can be satisfied with equality is called multilinear rank of $f$.
It determines the number of entries in $\mathcal{C}$ and the number of univariate functions needed to represent $f$.
For functions depending on more than three variables, a recursive format is usually preferable, leading to tree based formats~\cite{Falco18, Nouy17b, Nouy19} such as the hierarchical Tucker format~\cite{Hackbusch09, Schneider14} and the tensor train format \cite{Oseledets11,Bigoni16, Dektor19, Gorodetsky17, Gorodetsky19}.

The existence of a good approximation of the form~\eqref{eq:tuckerapprox} depends, in a nontrivial manner, on properties of $f$. 
It can be shown that the best approximation error in the format decays algebraically with respect to the multilinear rank of the approximation for functions in Sobolev \revision{spaces}~\cite{Griebel19, Schneider14} 
and geometrically for analytic functions~\cite{Hackbusch07, Trefethen17b}.
Approximations based on the Tucker format are highly anisotropic~\cite{Trefethen17c}, i.e., a rotation of the function may lead to a very different behavior of the approximation error. 
This can be partially overcome by adaptively subdividing the domain of the function as proposed, e.g., by Aiton and Driscoll~\cite{Aiton18}. 

The representation~\eqref{eq:tuckerapprox} is not yet practical because it involves continuous objects; a combination of low-rank tensor \emph{and} function approximation is needed.
Univariate functions can be approximated using barycentric Lagrange interpolation based on Chebyshev points \cite{Aurentz17, Berrut04, Higham04}.
This interpolation is fundamental to Chebfun~\cite{Driscoll14} - a package providing tools to perform numerical computations on the level of functions~\cite{Platte10}. 
Operations with these functions are internally performed by manipulating the Chebyshev coefficients of the interpolant~\cite{Battles04}. 

In Chebfun2~\cite{Townsend14,Townsend13}, a bivariate function $f(x,y)$ is approximated by applying Adaptive Cross Approximation (ACA) \cite{Bebendorf11}, which yields a low-rank approximation in terms of function fibers (e.g. $f(x_\star,y)$ for fixed $x_\star$ but varying $y$), and interpolating these fibers.
In Chebfun3, Hashemi and Trefethen~\cite{Hashemi17} extended these ideas to trivariate functions by recursively applying ACA, to first break down the tensor to function slices (e.g., $f(x_\star,y,z)$ for fixed $x_\star$) and then to function fibers.
As will be explained in Section~\ref{sec:Chebfun3Redundant}, this indirect approach via slice approximations typically leads to  redundant function fibers, which in turn involve unnecessary function evaluations. This is particularly problematic when the evaluation of the function is expensive, e.g. when each sample requires the solution of a partial differential equation (PDE); see Section~\ref{sec:NumericalResults} for an example. \revision{Whilst we only focus on function approximation throughout this work, the scope of Chebfun3 is wider, for instance it contains tools to perform numerical integration and differentiation.}

In this paper, we propose a novel algorithm aiming at computing the Tucker decomposition directly. Our algorithm is called Chebfun3F to emphasize that it is based on selecting the \emph{Fibers} $u_i,v_j,w_k$ in the Tucker approximation~\eqref{eq:tuckerapprox}. To compute a suitable core tensor, oblique projections based on Discrete Empirical Interpolation (DEIM)~\cite{Chaturantabut10} are used. We combine this approach with heuristics similar to the ones used in Chebfun3 for choosing the univariate discretization parameters adaptively and for the accuracy verification.

The remainder of this paper is structured as follows. 
In Section~\ref{sec:Theory}, we introduce and analyze the approximation format used in Chebfun3 and Chebfun3F.
In Section~\ref{sec:Chebfun3}, we briefly recall the approximation algorithm currently used in Chebfun3. Section~\ref{sec:Chebfun3F} introduces our novel algorithm Chebfun3F. Finally, in Section~\ref{sec:NumericalResults}, we perform numerical experiments to compare Chebfun3, Chebfun3F and sparse grid interpolation.

%% file: 2Theory.tex
\section{Chebyshev Interpolation and Tucker Approximation} \label{sec:Theory}

\subsection{Chebyshev Interpolation}\label{sec:ChebInterp}
Given a function $f:[-1,1]^3\to \R$, we consider an  approximation of the form
\begin{equation}\label{eq:ChebyshevInterpolantForm}
    f(x,y,z) \approx \tilde{f}(x,y,z) = \sum_{i=1}^{n_1} \sum_{j=1}^{n_2} \sum_{k=1}^{n_3} \mathcal{A}_{ijk} T_i(x) T_j(y) T_k(z),
\end{equation}
where $\mathcal{A}\in \R^{n_1 \times n_2 \times n_3}$ is the coefficient tensor and $T_k(x) = \cos((k\sRevision{-1})\cos^{-1}(x))$ denotes the $k$-th Chebyshev polynomial.

To construct~\eqref{eq:ChebyshevInterpolantForm}, we use (tensorized) interpolation. 
Let $\mathcal{T}_{ijk}=f(x_i^{(1)},x_j^{(2)},x_k^{(3)})$ denote the tensor containing all function values on the grid of Chebyshev points $x_k^{({\sRevision{\ell}})} = \cos(\revision{(k-1)}\pi/(n_{{\sRevision{\ell}}}-1)),\ k = 1,\dots,n_{{\sRevision{\ell}}}\sRevision{,\ \ell=1,2,3}$~\cite{Trefethen13}.
The coefficient tensor $\mathcal{A}$ is computed \sRevision{uniquely} from $\mathcal{T}$ using Fourier transformations.
We define the transformation matrices $F^{({\sRevision{\ell}})}\in \R^{n_{{\sRevision{\ell}}}\times n_{{\sRevision{\ell}}}}$ for ${\sRevision{\ell}}=1,2,3$ as in~\cite[Sec. 8.3.2.]{Mason02}
\begin{equation*}
F^{({\sRevision{\ell}})} = \frac{2}{n_{{\sRevision{\ell}}}}
\begin{pmatrix}
\frac{1}{4} T_1(x_1^{({\sRevision{\ell}})}) & \frac{1}{2} T_1(x_2^{({\sRevision{\ell}})}) & \frac{1}{2} T_1(x_3^{({\sRevision{\ell}})}) & \dots & \frac{1}{4} T_{1}(x_{n_{{\sRevision{\ell}}} }^{({\sRevision{\ell}})}) \\
\frac{1}{2} T_2(x_1^{({\sRevision{\ell}})}) & T_2(x_2^{({\sRevision{\ell}})}) &  T_2(x_3^{({\sRevision{\ell}})}) & \dots & \frac{1}{2} T_{2}(x_{n_{{\sRevision{\ell}}} }^{({\sRevision{\ell}})}) \\
\frac{1}{2} T_3(x_1^{({\sRevision{\ell}})}) &  T_3(x_2^{({\sRevision{\ell}})}) &  T_3(x_3^{({\sRevision{\ell}})}) & \dots & \frac{1}{2} T_{3}(x_{n_{{\sRevision{\ell}}} }^{({\sRevision{\ell}})})
\\ \vdots & \vdots & \vdots & \ddots & \vdots \\
\frac{1}{4} T_{n_{{\sRevision{\ell}}}}(x_1^{({\sRevision{\ell}})}) & \frac{1}{2} T_{n_{{\sRevision{\ell}}}}(x_2^{({\sRevision{\ell}})}) & \frac{1}{2} T_{n_{{\sRevision{\ell}}}}(x_3^{({\sRevision{\ell}})}) & \dots & \frac{1}{4} T_{n_{{\sRevision{\ell}}}}(x_{n_{{\sRevision{\ell}}} }^{({\sRevision{\ell}})})
\end{pmatrix}.
\end{equation*}
The mapping from the function evaluations to the coefficients can now be written as 
\begin{equation}\label{eq:MappingTtoA}
    \mathcal{A} = \mathcal{T} \times_1 F^{(1)} \times_2 F^{(2)} \times_3 F^{(3)},
\end{equation}
where $\mathcal{T} \times_{\sRevision{\ell}} M$ denotes the mode-${\sRevision{\ell}}$ multiplication. For a tensor $\mathcal{T} \in \R^{n_1 \times n_2 \times n_3}$ and a matrix $M \in \R^{m \times n_{\sRevision{\ell}}}$  it is defined as the multiplication of every mode-${\sRevision{\ell}}$ fiber of $\mathcal{T}$ with $M$, i.e.
\begin{equation*}
    \bra{\mathcal{T} \times_{\sRevision{\ell}} M}^{\set{{\sRevision{\ell}}}} = M \mathcal{T}^{\set{{\sRevision{\ell}}}},
\end{equation*}
where $\mathcal{T}^{\set{{\sRevision{\ell}}}}$ denotes the mode-${\sRevision{\ell}}$ matricization, which is the matrix containing all mode-${\sRevision{\ell}}$ fibers of $\mathcal{T}$~\cite{Kolda09}. By construction, the interpolation condition $\tilde{f}(x_i^{(1)},x_j^{(2)},x_k^{(3)}) = f(x_i^{(1)},x_j^{(2)},x_k^{(3)})$ is satisfied in all Chebyshev points.

The approximation error for Chebyshev interpolation applied to  multivariate analytic functions has been studied, e.g., by Sauter and Schwab in~\cite{Sauter11}. The following result states that the error decays exponentially with respect to the number of interpolation points in each variable.
\begin{theorem}[{\cite[Lemma 7.3.3.]{Sauter11}}]\label{thm:3DChebyshevInterpolationErrorAnalyticDecay}
Suppose that $f \in C([-1,1]^3)$ can be extended to an analytic function $f^*$ on $\mathcal{E}_{\rho} = E_{\rho_1} \times E_{\rho_2} \times E_{\rho_3}$ with $\rho_i>1$, where $E_\rho$ denotes the Bernstein ellipse, a closed ellipse with foci at $\pm 1$ and the sum of major and minor semi-axes equal to $\rho$.
Then the Chebyshev interpolant $\tilde{f}$ constructed above satisfies for $n=n_1=n_2=n_3$ the error bound
\begin{equation*}
     \norm{f- \tilde{f}}_{\infty} \leq 2^{2.5} \sqrt{3}   \rho_{\min}^{-n} \bra{1-\rho_{\min}^{-2}}^{\revision{-1.5}} \umax{z \in \mathcal{E}^{\rho}} \abs{f^*(z)},
\end{equation*}
where $\norm{\cdot}_\infty$ denotes the uniform norm on $[-1,1]^3$ and $\rho_{\min} = \min (\rho_1,\rho_2,\rho_3)$.
\end{theorem}

\sRevision{\textit{Remark.} In principle, approximation~\eqref{eq:ChebyshevInterpolantForm} has the desired format~\eqref{eq:tuckerapprox}. However, the size of the core tensor is linked to the polynomial degree $(n_1-1,n_2-1,n_3-1)$ of $\tilde{f}$. In the following, we introduce low-rank approximations, for which the size of the core tensor can be reduced.}

\subsection{Tucker Approximations}
A Tucker approximation of multilinear rank $(r_1,r_2,r_3)$  for a tensor $\mathcal{T} \in \R^{n_1 \times n_2 \times n_3}$ takes the form
\begin{equation*}%\label{eq:TuckerInitialDefinition}
    \mathcal{T} \approx \hat{\mathcal{T}} = \mathcal{C} \times_1 U \times_2 V \times_3 W,
\end{equation*}
where $\mathcal{C} \in \R^{r_1\times r_2 \times r_3}$ is called core tensor and $U \in \R^{n_1 \times r_1}$, $V \in \R^{n_2 \times r_2}$, $W \in \R^{n_3 \times r_3}$ are called factor matrices.
If $r_{\sRevision{\ell}} \ll n_{\sRevision{\ell}}$, the required storage is reduced from $n_1n_2n_3$ for $\mathcal{T}$ to $r_1r_2r_3+r_1n_1+r_2n_2+r_3n_3$ for $\hat{\mathcal{T}}$.

\subsection{Combining the Tucker Approximation and Chebyshev Interpolation} \label{sec:CombiningTuckerChebyshev}
Let $\hat{\T} = \mathcal{C} \times_1 U \times_2 V \times_3 W$ be a Tucker approximation of the tensor $\T$ obtained from evaluating $f$ in Chebyshev points.
Inserted into~\eqref{eq:ChebyshevInterpolantForm}, we now consider an approximation of the form
\begin{equation}\label{eq:ExplicitHatF}
    \hat{f}(x,y,z) = \sum_{i=1}^{n_1} \sum_{j=1}^{n_2} \sum_{k=1}^{n_3} \hat{\mathcal{A}}_{ijk} T_i(x) T_j(y) T_k(z),
\end{equation}
where the interpolation coefficients $\hat{\mathcal{A}}$ are computed from $\hat{\mathcal{T}}$ as in Equation~\eqref{eq:MappingTtoA}
\begin{align*} %\label{eq:LowRankInterpolationCoeffients}
    \hat{\mathcal{A}} &= \hat{\T} \times_1 F^{(1)} \times_2 F^{(2)} \times_3 F^{(3)} \\
    &= \mathcal{C} \times_1 F^{(1)}U \times_2 F^{(2)}V \times_3 F^{(3)}W.
\end{align*}
Note that the application of $F^{({\sRevision{\ell}})}$ is the mapping from function evaluations to interpolation coefficients in the context of univariate Chebyshev interpolation~\cite{Battles04, Mason02}. By interpreting the values stored in the columns of $U$ as function evaluations at Chebyshev points, we can define columnwise Chebyshev interpolant\sRevision{s} $\sRevision{u_j(x) = \sum_{i=1}^{n_1} (F^{(1)}U)_{ij} T_i(x)}$, \sRevision{$j = 1,\dots,r_1$}. \sRevision{Analogous interpolation based on $F^{(2)}V,F^{(3)}W$}, allows us to rewrite the approximation~\eqref{eq:ExplicitHatF} as
\begin{equation} \label{eq:FunctionalTuckerFormat}
    \hat{f}(x,y,z) = \sRevision{\sum_{i=1}^{r_1} \sum_{j=1}^{r_2} \sum_{k=1}^{r_3} \mathcal{C}_{ijk} u_i(x) v_j(y) w_k(z).}
\end{equation}
The goal of this paper is to compute approximations of this form. \revision{The algorithms presented in Sections~\ref{sec:Chebfun3} and~\ref{sec:Chebfun3F} internally compute the underlying $\hat{\T}$.}

\subsection{Low-Rank Approximation Error} 

The following lemma allows us to distinguish the interpolation error\sRevision{, which can be bounded using Theorem~\ref{thm:3DChebyshevInterpolationErrorAnalyticDecay},} from the low-rank approximation error in the approximation~\eqref{eq:FunctionalTuckerFormat}.
\begin{lemma} \label{lem:ErrorBoundF-FHat}
Consider the Chebyshev interpolation $\tilde{f}$ defined in~\eqref{eq:ChebyshevInterpolantForm} and let $\hat{\T}$ be an approximation of the involved function evaluation tensor $\T \in \R^{n_1\times n_2\times n_3}$. Then the approximation $\hat{f}$ defined in~\eqref{eq:FunctionalTuckerFormat} satisfies
\begin{align}\label{eq:ErrorSplitting}
\norm{f-\hat{f}}_{\infty} \leq& \norm{f-\tilde{f}}_{\infty} + {\bra{\frac{2}{\pi}\log(n_1\sRevision{-1})+1} \bra{\frac{2}{\pi}\log(n_2\sRevision{-1})+1} \bra{\frac{2}{\pi}\log(n_3\sRevision{-1})+1}} \norm{\T - \hat{\T}}_{\infty},
\end{align} 
\sRevision{where $\norm{\cdot}_\infty$ denotes the uniform norm on $[-1,1]^3$ for functions, and the maximum norm on $\R^{n_1\times n_2 \times n_3}$ for tensors.}
\end{lemma}
\begin{proof}
By applying the triangle inequality we obtain
\begin{align*}
\norm{f-\hat{f}}_{\infty}  \leq& \norm{f-\tilde{f}}_{\infty} + \norm{\tilde{f}-\hat{f}}_{\infty}.
\end{align*} 
\sRevision{The function $g =\tilde{f}-\hat{f}$ is the unique polynomial of degree $(n_1-1,n_2-1,n_3-1)$ satisfying $g(x_i^{(1)},x_j^{(2)},x_k^{(3)}) = (\T-\hat{\T})_{ijk}$ for $1\leq i \leq n_1,1\leq j \leq n_2, 1 \leq k \leq n_3$.
For univariate polynomial interpolation the Lebesgue constant $\Lambda_n$ bounds the ratio of uniform norm of the interpolant and the maximum absolute value at the $n+1$ interpolation nodes. For univariate Chebyshev interpolation we have $\Lambda_{n} \leq (2/\pi) \log({n})+1$~\cite{Trefethen13}. This generalizes to tensorized interpolation via the product of the Lebesgue constants for univariate interpolation~\cite{Mason80},
\[ 
\norm{\tilde f - \hat{ f}}_\infty = \norm{g}_ \infty \leq \Lambda_{n_1-1}\Lambda_{n_2-1}\Lambda_{n_3-1}\norm{\T- \hat \T}_\infty.
\]
}
\end{proof}

Lemma~\ref{lem:ErrorBoundF-FHat} states that $\hat{f}$ is nearly as accurate as $\tilde{f}$ when the error bound~\eqref{eq:ErrorSplitting} is not dominated by $\norm{\T-\hat{\T}}_\infty$.
The low-rank approximation $\hat{f}$ can be stored more efficiently than $\tilde{f}$ when $r_{\sRevision{\ell}} \ll n_{\sRevision{\ell}}$.
In the next section, we provide some insight into an example that features $r_{\sRevision{\ell}} \ll n_{\sRevision{\ell}}$.

\subsection{When the Low-Rank Approximation is More Accurate}
\label{sec:RankDegreeExample}
We consider the function 
\begin{equation*}
f_\eps(x,y,z) = \frac{1}{x+y+z+3+\eps}
\end{equation*} 
on $[-1,1]^3$ with parameter $\eps>0$.
Let $\tau \geq 0$.

In this section, we show that a Chebyshev interpolation $\tilde f_\eps$ satisfying $\norm{\tilde{f}_\eps-f_\eps}_\infty \leq \tau$ for a prescribed error bound $\tau$ requires polynomial degrees $n_{\sRevision{\ell}} = \mathcal{O}(1/\log(1+\sqrt{\eps}))$. However, one can achieve $\norm{\T-\hat{\T}}_\infty \leq \tau$ with multilinear ranks $r_{\sRevision{\ell}} \leq \mathcal{O(\abs{\log(\eps))}}$, \revision{which grows much slower than $\mathcal{O}(1/\log(1+\sqrt{\eps})) \approx \mathcal{O}(\eps^{-1/2})$ for $\eps \to 0$}. Therefore, for small values of $\eps$ the required polynomial degree \revision{$(n_1,n_2,n_3)$} is much higher than the required multilinear rank \revision{$(r_1,r_2,r_3)$}.
In this situation, $\hat{f}_\eps$ can achieve almost the same accuracy as $\tilde{f}_\eps$, but with significantly less storage.

\subsubsection*{Polynomial Degree}
For the degree $(n_1,n_2,n_3)$ Chebyshev interpolant~$\tilde{f}_\eps$ we require $\norm{f_\eps-\tilde{f}_\eps}_\infty \leq \tau$, which is equivalent to $\norm{\eps(f_\eps-\tilde{f}_\eps)}_\infty \leq \eps\tau$.
By Theorem~\ref{thm:3DChebyshevInterpolationErrorAnalyticDecay},
\begin{equation*}
\norm{\eps(f_\eps-\tilde{f}_\eps)}_\infty \leq \mathcal{O}(\rho^{-n}_{\min} \cdot \umax{\revision{(\sRevision{x,y},z)} \in \mathcal{E}_{\rho}} |\eps f_\eps^*(\revision{x,y,z})|).
\end{equation*}
We set $\rho_1 = \rho_2 = \rho_3 = 1 + \eps/6 + \sqrt{(1+\eps/6)^2 -1}$ and extend $\eps f_\eps$ analytically to  $\eps  f_\eps^*(x,y,z) = \eps \bra{x+y+z+3+\eps}^{-1}$ on $\mathcal{E}_{\rho}$.
By construction $\umax{\revision{(x,y,z)} \in \mathcal{E}_{\rho}} |\eps f_\eps^*(x,y,z)| = 2$ \revision{is assumed for $x=y=z=-1-\eps/6$, where $|x+y+z+3+\eps|$ is minimized}.
Hence, we can choose $n_1 = n_2 = n_3 = \mathcal{O}(1/\log(1+\sqrt{\eps}))$ to obtain the desired accuracy.
Although this is only an upper bound for the polynomial degree required, numerical experiments reported below indicate that it is tight.

\subsubsection*{Multilinear Rank} 

An a priori approximation with exponential sums is used to obtain a bound on the multilinear rank for a tensor containing function values of $f_\eps$; see~\cite{Hackbusch19}. Given $R> 1$ and $r\in \mathbb N$, Braess and Hackbusch~\cite{Braess05} showed that there exist coefficients $a_i$ and $b_i$ such that 
\begin{equation} \label{eq:ExpSumBestApproxBound}
\Big| \frac1x - \sum_{i=1}^r a_i \exp(-b_i x) \Big| \leq 16 \exp \bra{- \frac{r \pi^2}{\log(8R)}}, \quad \forall x\in[1,R].
\end{equation}
Trivially, we have $\eps f_\eps(x,y,z) = 1/\omega$ for the substitution $\omega = (x+y+z+3+\eps)/\eps$ with $\omega \in [1, 1 + 6/\eps]$.
Applying~\eqref{eq:ExpSumBestApproxBound} yields \revision{that there exist $a_i$ and $b_i$ such that}
$\abs{1/\omega - \revision{\sum_{i=1}^r a_i \exp(-b_i} \omega)} \leq \tau \eps$ or, equivalently,
\begin{equation} \label{eq:ExpSumError3D}
\bigg\lVert
f_\eps(x,y,z) - \underbrace{\sum_{i=1}^r \frac{a_i}{\eps} \cdot \exp\bra{-\frac{b_i}{\eps}x} \cdot \exp\bra{-\frac{b_i}{\eps}y} \cdot \exp\bra{-\frac{b_i}{\eps}z} \cdot \exp\bra{-\frac{b_i}{\eps}(3+\eps)}}_{=:g_\eps(x,y,z)}
\bigg\rVert_\infty \leq \tau
\end{equation}
for every $x,y,z\in [-1,1]$ when
\begin{equation*}
r \geq \frac{-\log\bra{8\bra{1+\frac{6}{\eps}}}\log\bra{\frac{\tau}{16}}}{\pi^2} = \mathcal{O}(\abs{\log(\eps)}).
\end{equation*}
The approximation $g_\eps$ in~\eqref{eq:ExpSumError3D} has multilinear rank $(r,r,r)$. 
In turn, the tensor $\hat{\T}_{ijk}=g_\eps(x_i^{(1)},x_j^{(2)},x_k^{(3)})$ has multilinear rank at most $(r,r,r)$ and satisfies $\norm{\T-\hat{\T}}_\infty \leq \tau$.

\subsubsection*{Comparison}
In Figure~\ref{fig:RankDegStudy}, we estimate the maximal polynomial degree 
required to compute a Chebyshev interpolant with accuracy $\tau=10^{-10}$ for selected fibers of $f_\eps$, which is a lower bound for the required polynomial degrees $n_{{\sRevision{\ell}}}$.
It perfectly matches the asymptotic behavior of the derived upper bound $\mathcal{O}(1/\log(1+\sqrt{\eps}))$.
In Figure~\ref{fig:RankDegStudy}, we also plot the multilinear ranks from the truncated Higher Order Singular Value Decomposition (HOSVD)~\cite{Lathauwer00} with tolerance $\tau$ applied to the tensor containing the evaluation of $f_\eps$ on a $150\times150\times150$ Chebyshev grid. This estimate serves as a lower bound for the multilinear rank required to approximate $f_\eps$. Due to the limited grid size, this estimate does not fully match the asymptotic behavior $\abs{\log(\eps)}$, but nonetheless it clearly reflects that the multilinear ranks can be much smaller than the polynomial degrees, as predicted by
$\abs{\log(\eps)} \ll 1/\log(1+\sqrt{\eps})$, for sufficiently small $\eps$.

\begin{figure}[!ht]
\centering
\includegraphics[width=0.4\textwidth]{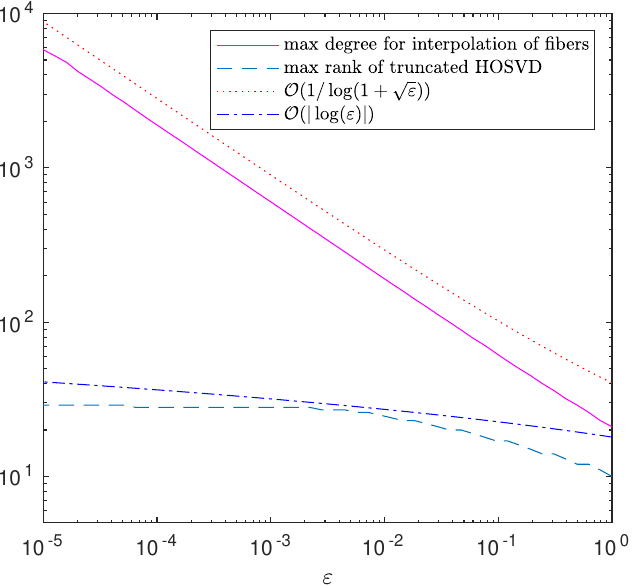}
\caption{\revision{Comparison of the} theoretical upper bounds $\mathcal{O}(\abs{\log(\eps)})$ for the multilinear rank and $\mathcal{O}(1/\log(1+\sqrt{\eps}))$ for the polynomial degree for varying $\eps$\revision{, the} maximal polynomial degree, which is used by Chebfun to approximate selected functions fibers of $f_\eps$ up to accuracy $10^{-10}$\revision{, and the} maximal multilinear rank of the truncated HOSVD with tolerance $10^{-10}$ of a sample tensor on a $150\times150\times150$ Chebyshev grid. The constants in the bounds are chosen to result in curves close to the data.
}
\label{fig:RankDegStudy}
\end{figure}

%% file: 3Chebfun3.tex
\section{Existing Algorithm: Chebfun3}\label{sec:Chebfun3}

In this section, we recall how an approximation of the \revision{form~\eqref{eq:FunctionalTuckerFormat}} is computed in Chebfun3~\cite{Hashemi17}. 
As discussed in Section~\ref{sec:RankDegreeExample}, there are often situations in which the multilinear rank of $\hat{f}$ is much smaller than the polynomial degree. Chebfun3 benefits from such a situation by first using a coarse sample tensor $\T_c$ to identify the fibers needed for the low-rank approximation. 
This allows to construct the actual approximation from a finer sample tensor $\T_f$ by only evaluating these fibers instead of the whole tensor. 

\revision{Chebfun3} consists of three phases: preparation of the approximation by identifying fibers for a so called block term decomposition~\cite{Lathauwer08b} of $\T_c$, refinement of the fibers, conversion and compression of the refined block term decomposition into Tucker format~\eqref{eq:FunctionalTuckerFormat}.

\subsection{Phase~1: Block Term Decomposition} 
In Chebfun3, $\T_c \in \R^{n_1^{(c)}\times n_2^{(c)} \times n_3^{(c)}}$  is initially obtained by sampling $f$ on a $17\times17\times17$ grid of Chebyshev points.
A block term decomposition of $\T_c$ is obtained by applying ACA~\cite{Bebendorf11} (see Algorithm~\ref{alg:ACA}) recursively.
In the first step, ACA is applied to a matricization of $\T_c$, say, the mode-$1$ matricization $\T_c^{\set{1}}$. This results in index sets $I,J$ such that
\begin{align}\label{eq:InitialACACF3}
\T_c^{\set{1}} \approx \T_c^{\set{1}}(:,J)\bra{\T_c^{\set{1}}(I,J)}^{-1}\T_c^{\set{1}}(I,:),
\end{align}
where $\T_c^{\set{1}}(:,J)$ contains mode-$1$ fibers of $\T_c$ and $\T_c^{\set{1}}(I,:)$ contains mode-$(2,3)$ slices of $\T_c$.
For each $i \in I$, such a slice $\T_c^{\set{1}}(i,:)$ is reshaped into a matrix $S_i = \T_c(i,:,:) \in \R^{n_2^{(c)}\times n_3^{(c)}}$ and, in the second step, approximated by again applying ACA:
\begin{equation} \label{eq:RecACACF3}
S_i \approx S_i(:,L_i)\bra{S_i(K_i,L_i)}^{-1}S_i(K_i,:),
\end{equation}
where $S_i(:,L_i)$ and $S_i(K_i,:)$ contain mode-$2$ and mode-$3$  fibers of $\T_c$, respectively.
Combining \eqref{eq:InitialACACF3} and \eqref{eq:RecACACF3} yields the approximation 
\begin{equation}\label{eq:BlockTermDecomposition}
    \T_c^{\set{1}} \approx \T_c^{\set{1}}(:,J)\bra{\T_c^{\set{1}}(I,J)}^{-1}
    \begin{pmatrix}
    \mathsf{vec}(S_1(:,L_1)\bra{S_1(K_1,L_1)}^{-1}S_1(K_1,:)) \\
    \mathsf{vec}(S_2(:,L_2)\bra{S_2(K_2,L_2)}^{-1}S_2(K_1,:)) \\
    \vdots
    \end{pmatrix},
\end{equation}
where $\mathsf{vec}$ denotes vectorization. Reshaping this approximation into a tensor can be viewed as a block term decomposition in the sense of~\cite[Definition 2.2.]{Lathauwer08b}.

If the ratios of $\abs{I}/n_1^{(c)}$, $\abs{K_i}/n_2^{(c)}$ and $\abs{L_i}/n_3^{(c)}$ are larger than the heuristic threshold $(2\sqrt{2})^{-1}$ the coarse grid resolution $(n_1^{(c)},n_2^{(c)},n_3^{(c)})$ is deemed insufficient to identify fibers.
If this is the case, $n_{\sRevision{\ell}}^{(c)}$ is increased to $\big\lfloor\sqrt{2}^{\floor{2\log_2(n_{{\sRevision{\ell}}}^{(c)}) + 1}}\big\rfloor + 1$ and Phase~1 is repeated.

\begin{algorithm}
\caption{ACA}\label{alg:ACA} 
\begin{algorithmic}[1]
\State \textbf{Input:} matrix $M$, tolerance $\eps$
\State $I$ = [], $J$ = []
\While $\max \abs{M} \geq \eps$
\State $(i,j) = \uargmax{(i,j)} \abs{M(i,j)}$
\State $I = [I,i]$, $J = [J,j]$
\State $M = M - M(:,j)M(i,:)/M(i,j)$
\EndWhile
\State \textbf{Output:} index sets $I$ and $J$ s.t. $M \approx M(:,J)M(I,J)^{-1}M(I,:)$
\end{algorithmic}
\end{algorithm} 

\subsection{Phase~2: Refinement} \label{sec:curPhase2}

The block term decomposition~\eqref{eq:BlockTermDecomposition} is composed of fibers of $\T_c$. 
Such a fiber $\T_c(:,j,k)$ corresponds to the evaluation of a univariate function $f(\cdot,y,z)$ for certain fixed $y,z$.
Chebfun contains a heuristic to decide whether the function values in $\T_c(:,j,k)$ suffice to yield an accurate interpolation of $f(\cdot,y,z)$~\cite{Aurentz17}. If this is not the case, the grid is refined.

In Chebfun3 this heuristic is applied to all fibers contained in~\eqref{eq:BlockTermDecomposition} in order to determine the size $n_1^{(f)} \times n_2^{(f)} \times n_3^{(f)}$, initially set to $n_1^{(c)} \times  n_2^{(c)} \times  n_3^{(c)}$, of the finer sample tensor $\T_c$. For each ${\sRevision{\ell}} \in \{1,2,3\}$, the size $n_{\sRevision{\ell}}$ is repeatedly increased by setting $n_{{\sRevision{\ell}}}^{(f)} := 2 n_{{\sRevision{\ell}}}^{(f)}-1$, which leads to nested Chebyshev points, until the heuristic considers the resolution sufficient for all mode-${\sRevision{\ell}}$ fibers.
Replacing all fibers in~\eqref{eq:BlockTermDecomposition} by their refined counterparts yields an approximation of the tensor $\T_f$, which contains evaluations of $f$ on a $n_1^{(f)}\times n_2^{(f)} \times n_3^{(f)}$ Chebyshev grid. Note that $\T_f$ might be very large and is never computed explicitly.

\subsection{Phase~3: Compression} %\label{sec:curPhase3}

In the third phase of the Chebfun3 constructor, the refined block term decomposition is converted and compressed to the desired Tucker format~\eqref{eq:FunctionalTuckerFormat}, where the interpolants $u_i(x)$ are stored as Chebfun objects~\cite{Battles04}; see~\cite{Hashemi17} for details. 
Lemma~\ref{lem:ErrorBoundF-FHat} guarantees a good approximation \revision{$\hat{f}$} when the polynomial degrees $n_1^{(f)}, n_2^{(f)}, n_3^{(f)}$ are sufficiently large and when $\T_f$ is well approximated by the \revision{underlying} Tucker approximation \revision{$\hat \T$}. 
Neither of these properties can be guaranteed in Phases 1 and 2 alone.
Therefore in a final step, Chebfun3 verifies the accuracy by comparing $f$ and the approximation $\hat f$ at Halton points~\cite{Niederreiter92}.
If the estimated error is too large, the whole algorithm is restarted on a finer coarse grid from Phase~1.

\subsection{Disadvantages}\label{sec:Chebfun3Redundant}

The Chebfun3 algorithm often requires unnecessarily many function evaluations. As we will illustrate in the following, this is due to redundancy among the mode-$2$ and mode-$3$ fibers.
For this purpose we collect all (refined) mode-$2$ fibers $S_i(:,L_i)$ in the block term decomposition~\eqref{eq:BlockTermDecomposition} into the columns of a big matrix $V^{\mathsf{BTD}}_m = \begin{bmatrix}S_1(:,L_1) & \cdots & S_{m}(:,L_{m})\end{bmatrix}$, where $m$ is the number of steps of the outer ACA~\eqref{eq:InitialACACF3}.
As will be demonstrated with an example below, matrix $V^{\mathsf{BTD}}_m$ is often observed to have low numerical rank, which in turn allows to represent its column space by much fewer columns, that is, much fewer mode-$2$ fibers. As the accuracy of the column space determines the accuracy of the Tucker decomposition after the compression, this implies that the other mode-$2$ fibers in $V^{\mathsf{BTD}}_m$ are redundant.

Let us now consider the block term decomposition\footnote{Note that the accuracy verification in Phase~3 fails once for this function. Here we only consider to block term decomposition obtained after restarting the procedure.}~\eqref{eq:BlockTermDecomposition} for the function 
\begin{equation*}%\label{eq:TestFunction}
f(x,y,z) = \frac{1}{1+25\sqrt{x^2+y^2+z^2}}.
\end{equation*}
In Figure~\ref{fig:RedundantFibers} the numerical rank and the number of columns of $V^{\mathsf{BTD}}_m$ are compared. For $m=10$ the approximation of the slices $S_{i_1}$ to $S_{i_{10}}$ leads to a total of $153$ mode-$2$ fibers, the sum of the corresponding red and blue bars in Figure~\ref{fig:RedundantFibers}. In contrast, their numerical rank (blue bar) is only $19$. Thus, the red bar can be interpreted as number of redundant mode-$2$ fibers. This happens since nearby slices tend to be similar.
\begin{figure}[ht]
\centering
\includegraphics[width=0.4\textwidth]{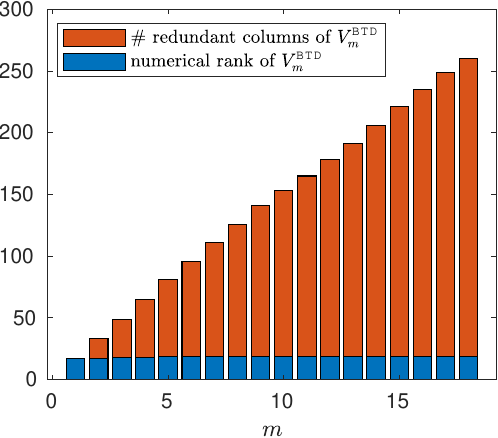}\\
\caption{Numerical rank and number of redundant columns (= total number of columns - numerical rank) of the matrix $V^{\mathsf{BTD}}_m$, whose columns are given by the refined version of the mode-$2$ fibers determined after $m$ steps of the outer ACA~\eqref{eq:InitialACACF3} in Chebfun3.}
\label{fig:RedundantFibers}
\end{figure}
The total block term decomposition contains $18$ slices and is compressed into a Tucker decomposition with multilinear rank $(17,19,19)$. It contains $242$ redundant fibers, the refinement requires $192$ function evaluations for each of them. Note that the asymmetry in the rank of the Tucker decomposition is caused by the asymmetry of the block term decomposition.

Another disadvantage is that Chebfun3 always requires the full evaluation of $\T_c$ in Phase~1.
This becomes expensive when a large size $n_1^{(c)} \times n_2^{(c)} \times n_3^{(c)}$ is needed in order to properly identify suitable fibers.

%% file: 4Chebfun3F.tex
\section{Novel Algorithm: Chebfun3F}\label{sec:Chebfun3F}
In this section, we describe our novel algorithm Chebfun3F to compute an approximation of the form~\eqref{eq:FunctionalTuckerFormat}. The goal of Chebfun3F is to the avoid the redundant function evaluations observed in Chebfun3. While the structure of Chebfun3F is similar to Chebfun3, consisting of 3~phases to identify/refine fibers and compute a Tucker decomposition, there is a major difference in Phase~1. Instead of proceeding via slices, we directly identify mode-${\sRevision{\ell}}$ fibers of $\T_c$ for building factor matrices. The core tensor is constructed in Phase~3.

\subsection{Phase~1: Fiber Indices and Factor Matrices}
As in Chebfun3, the coarse tensor $\T_c \in \R^{n_1^{(c)}\times n_2^{(c)} \times n_3^{(c)}}$  is initially defined to contain the function values of $f$ on a $17\times17\times17$ Chebyshev grid.
% \revision{$\T_c(i,j,k) = f(x_i^{(1)},x_j^{(2)},x_k^{(3)})$, $i,j,k=1,\ldots,17$.}
We seek to compute \sRevision{full rank} factor matrices $U_c \in \R^{n_1^{(c)} \times r_1}$, $V_c \in \R^{n_2^{(c)}\times r_2}$ and $W_c \in \R^{n_3^{(c)}\times r_3}$ such that the orthogonal projection of $\T_c$ onto the span of the factor matrices is an accurate approximation of $\revision{\T}_c$, i.e.
\begin{equation}\label{eq:CoarseTensorApproximation}
\T_c \approx \T_c \times_1 U_c (U_c^T U_c)^{-1} U_c^T \times_2 V_c (V_c^T V_c)^{-1} V_c^T \times_3 W_c (W_c^T W_c)^{-1} W_c^T.
\end{equation}
Additionally, we require that the columns in $U_c,V_c,W_c$ contain fibers of $\T_c$.

In the existing literature, algorithms to compute such factor matrices include the Higher Order Interpolatory Decomposition~\cite{Saibaba16}, which is based on a rank revealing QR~decomposition, and the Fiber Sampling Tensor Decomposition~\cite{Caiafa10}, which is a generalization of the CUR~decomposition. 
We propose a novel algorithm, which in contrast to the existing algorithms does not require the evaluation of the full tensor $\T_c$.
We follow the ideas of TT-cross~\cite{Oseledets10, Savostyanov11} and its variants such as the Schur-Cross3D~\cite{Rakhuba15} and the ALS-cross~\cite{Dolgov19b}. 

Initially, we randomly choose index sets $\tI,\tJ,\tK$ \revision{ by partitioning $\{1,\dots,17\}$ into $6$ subsets and sampling one index from each subset for each index set.}
In the first step, we apply Algorithm~\ref{alg:ACA} to $(\T_c(:,\tJ,\tK))^{\set{1}}$.
Note that this needs drawing only $36\cdot n_1^{(c)}$ values of the function $f$, in contrast to $n_1^{(c)} n_2^{(c)} n_3^{(c)}$ values in the whole tensor $\T_c$.
The selected $r_1$ columns serve as a first candidate for the factor matrix $U_c$. The index set $\tI$ is set to the row indices selected by Algorithm~\ref{alg:ACA} (see Figure~\ref{fig:IndexVisualization}).
We use the updated index set and apply Algorithm~\ref{alg:ACA} to $(\T_c(\tI,:\tK))^{\set{\revision{2}}}$ analogously, which yields $V_c$ and an updated $\tJ$. From $(\T_c(\tI,\tJ,:))^{\set{\revision{3}}}$ we obtain $W_c$ and $\revision{\tK}$.
We repeat this process in an alternating fashion with the updated index sets, which leads to potentially improved factor matrices.
Following the ideas of Chebfun3, we check after each iteration whether the ratios $r_1/n_1^{(c)}$, $r_2/n_2^{(c)}$ and $r_3/n_3^{(c)}$ surpass the heuristic threshold $(2\sqrt{2})^{-1}$. 
If this is the case, we increase the size of the coarse tensor $n_{\sRevision{\ell}}^{(c)}$ to $\big\lfloor\sqrt{2}^{\floor{2\log_2(n_{{\sRevision{\ell}}}^{(c)}) + 1}}\big\rfloor + 1$ and restart the whole process by reinitializing $\tI, \tJ, \tK$ with $r_1,r_2,r_3$ random indices respectively.
\begin{figure}
    \centering
    \includegraphics{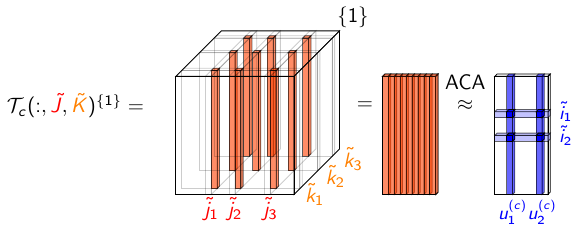}
    \caption{Visualization of applying ACA (Algorithm~\ref{alg:ACA}) to a matricization of a subtensor.}
    \label{fig:IndexVisualization}
\end{figure} 

It is not clear a priori how many iterations are needed to attain an approximation~\eqref{eq:CoarseTensorApproximation} that yields a Tucker approximation~\eqref{eq:FunctionalTuckerFormat} which passes the accuracy verification in Phase~3.
In numerical experiments, it has usually proven to be sufficient to stop after the second iteration, during which the coarse grid has not been refined, or when $\abs{\tI}\leq1$, $\abs{\tJ}\leq1$ or $\abs{\tK}\leq1$. This is formalized in Algorithm~\ref{alg:Phase1}. \sRevision{Note that $U_c,V_c,W_c$ are full rank by construction, since Algorithm~\ref{alg:ACA} stops based on the tolerance $\eps \geq 0$.} In many cases, we found that the numbers of columns in the factor matrices are equal to the multilinear rank of the truncated HOSVD~\cite{Lathauwer00} of $\T_c$ with the same tolerance. 

\begin{algorithm}[ht]
\caption{Factor Matrix Computation}\label{alg:Phase1}
\begin{algorithmic}[1]
\State \textbf{Input:} $f$, $(n_1^{(c)}, n_2^{(c)}, n_3^{(c)})$, $(r_1,r_2,r_3)$ 
\State Let $\revision{\T_c}(i,j,k)$ = $f(x_i^{(1)},x_j^{(2)},x_k^{(3)})$ be a function that evaluates the values of $f$ on a $n_1^{(c)}\times n_2^{(c)}\times n_3^{(c)}$ Chebyshev grid on demand.
\State initialize $\tJ,\tK$ with $r_2,r_3$ randomly chosen indices in $\big\{1,\dots,n_1^{(c)}\big\},\big\{1,\dots,n_2^{(c)}\big\}$
\For \text{iterations} = 1:2
\State compute ACA of $\T_c(:,\tilde{J},\tilde{K})^{\set{1}}$ $\rightarrow$ $U_c \in \R^{n_1^{(c)} \times r_1}$ = selected columns, $\tI$ = selected row indices
\State compute ACA of $\T_c(\tI,:,\tK)^{\set{2}}$ $\rightarrow$ $ V_c \in \R^{n_2^{(c)} \times r_2}$ = selected columns, $\tJ$ = selected row indices
\State compute ACA of $\T_c(\tI,\tJ,:)^{\set{3}}$ $\rightarrow$ $W_c \in \R^{n_3^{(c)} \times r_3}$ = selected columns, $\tK$ = selected row indices
\State \textbf{if} the multilinear ranks get too large $\rightarrow$ adjust the size of $(n_1^{(c)},n_2^{(c)},n_3^{(c)})$ and go to line~2
\State \textbf{if} $r_1\leq 1$ or $r_2\leq 1$ or $r_3\leq 1$ $\rightarrow$ go to line~10
\EndFor
\State \textbf{Output:} $U_c,V_c,W_c$, $(n_1^{(c)}, n_2^{(c)}, n_3^{(c)})$, $(r_1,r_2,r_3)$
\end{algorithmic}
\end{algorithm}

\subsection{Phase~2: Refinement of the Factors}
In Phase~2, the fibers in $U_c,V_c,W_c$ are refined using Chebfun's heuristic~\cite{Aurentz17} as in Chebfun3 (see Section~\ref{sec:curPhase2}). 
This leads to \sRevision{full rank} factor matrices $U_f \in \R^{n_1^{(f)}\times r_1}$, $V_f\in \R^{n_2^{(f)}\times r_2}$ and $W_f \in \R^{n_3^{(f)}\times r_3}$  containing the refined fibers of $\T_f$,  corresponding to the evaluations of $f$ on a $n_1^{(f)} \times n_2^{(f)} \times n_3^{(f)}$ Chebyshev grid,
\revision{$\T_f(i,j,k) = f(x_i^{(1)},x_j^{(2)},x_k^{(3)})$, $i=1,\ldots,n_1^{(f)}$, $j=1,\ldots,n_2^{(f)}$, $k=1,\ldots,n_3^{(f)}$.}
This phase needs only $\mathcal{O}(\sum_{{\sRevision{\ell}}}n_{{\sRevision{\ell}}}^{(f)}r_{{\sRevision{\ell}}})$ evaluations of $f$.

\subsection{Phase~3: Reconstruction of the Core Tensor} %\label{sec:Phase3New}
In the final Phase of Chebfun3F, we compute a core tensor $\hat{\mathcal{C}}$ to yield an approximation $\T_f \approx \hat{\mathcal{C}} \times_1 U_f \times_2 V_f \times_3 W_f$.

In principle, the best approximation (with respect to the Frobenius norm) for fixed factor matrices $U_f,V_f,W_f$ is obtained by orthogonal projections \cite{Lathauwer00}.
Such an approach comes with the major disadvantage that the full evaluation of $\T_f$ is required.
This can be circumvented by instead using oblique projections\revision{. The oblique projection onto the span of $U_f$ is defined as} $U_f (\Phi_I^T U_f)^{-1} \Phi_I^T$, where  $\Phi_I^T U_{\revision{f}} = U_{\revision{f}}(I,:)$ for an index set $I$ \revision{which contains $r_1$ indices selected from} $\{1,\dots,n_1^{(f)}\}$. \revision{Analogous o}blique projections in all three modes yield
\begin{equation*}
    \T_f \approx \hat{\T} =  \underbrace{\bra{\revision{\underbrace{(\T_f \times_1 \Phi_I^T \times_2 \Phi_J^T \times_3 \Phi_K^T)}_{=\T_f(I,J,K)}} \times_1 (\Phi_I^T U_f)^{-1} \times_2 (\Phi_J^T V_f)^{-1} \times_3 (\Phi_K^T W_f)^{-1}}}_{= \hat{\mathcal{C}}} \times_1 U_f  \times_2 V_f  \times_3 W_f,
\end{equation*}
for index sets $I,J,K$. The choice of $I,J,K$ is crucial for the approximation quality and will be discussed later on.
\revision{Note that the computation of the $\T_f$ only requires $r_1r_2r_3$ additional evaluations of $f$.}
From $\hat{\T}$ we construct the approximation~\eqref{eq:FunctionalTuckerFormat} as described in Section~\ref{sec:CombiningTuckerChebyshev}. 

\revision{
Let $Q_U,Q_V,Q_W$ denote the orthogonal matrices in the QR decompositions of $U_f, V_f, W_f$. Note that $U_f (\Phi_I^T U_f)^{-1} = {Q_U(\Phi_I^T Q_U)^{-1}}$ and $(\T_f(I,J,K) \times_1 \Phi_I^T U_f)^{-1}) \times_1 U_f = \T_f(I,J,K) \times_1 U_f (\Phi_I^T U_f)^{-1}$. In Chebfun3F, we treat $\hat{\T}$ as Tucker decomposition of the form}
\begin{align}\label{eq:CF3FThat}
\hat{\T}
=& {\T_f(I,J,K)} \times_1 {Q_U(\Phi_I^T Q_U)^{-1}} \times_2 {Q_V(\Phi_J^T Q_V)^{-1}} \times_3 {Q_W (\Phi_K^T Q_W)^{-1}}
\end{align}
\revision{to avoid the potentially ill-conditioned matrices $(\Phi_I^T U_f)^{-1},(\Phi_J^T V_f)^{-1},(\Phi_K^T W_f)^{-1}$. } \sRevision{Note that we have $\hat{T}(I,J,K) = T(I,J,K)$ by construction.}

The following lemma plays a critical role in guiding the choice of indices $I,J,K$.
\begin{lemma}[{\cite[Lemma 7.3]{Chaturantabut10}}]\label{lem:DEIMerror}
Let $M \in \R^{n\times m}$, $n\geq m$, have orthonormal columns.
Consider an index set $I\subset \set{1,\dots,n}$ of cardinality $m$ such that $\Phi_I^T M$ is invertible. Then the oblique projection $M (\Phi_I^T M)^{-1} \Phi_I^T$ satisfies
\begin{equation*}
\norm{x- M (\Phi_I^T M)^{-1} \Phi_I^T x}_2 \leq \norm{(\Phi_I^TM)^{-1}}_2\cdot\norm{(I-MM^T)x}_2, \quad \forall x \in \R^{n},
\end{equation*}
where $\norm{\cdot}_2$ denotes the matrix $2$-norm.
\end{lemma}

Lemma~\ref{lem:DEIMerror} exhibits the critical role played by the quantity $\norm{(\Phi_I^T Q_U)^{-1}}_2 \ge 1$ for oblique projections. In \revision{Chebfun3F}, we use the \textit{discrete empirical interpolation method} (DEIM) \cite{Chaturantabut09}, presented in Algorithm~\ref{alg:DEIM}, to compute the index sets $I,J,K$ given $Q_U,Q_V,Q_W$. In practice, these index sets usually yield good approximations as $\norm{(\Phi_I
^TQ_U)^{-1}}_2$ tends to be small; see also Section~\ref{sec:Chebfun3Ferror}.

\begin{algorithm}[ht]
\caption{Discrete Empirical Interpolation Method}\label{alg:DEIM}
\begin{algorithmic}[1]
\State \textbf{Input:} orthonormal matrix $M \in \mathbb{R}^{n\times m}$  
\State $I = [\mathsf{argmax}\ |M(:,1)|]$
\For $k = 2,\dots,m$
\State $c = M(I,1:k-1) \setminus M(I,k)$
\State $r = M(:,k) - M(:,1:k-1)c$
\State $I = [I,\mathsf{argmax}\ |r|]$
\EndFor
\State \textbf{Output:} index set $I$
\end{algorithmic}
\end{algorithm}

\subsection{Chebfun3F Algorithm}
% Given the tensor approximation $\hat{\T}$,
\revision{Having computed the Tucker factors $U_f$,  $V_f$ and $W_f$ in Phase 2, and the core $\hat{\mathcal{C}}$ in Phase 3,}
we obtain $\hat{f}$ by interpolating the factor matrices as described in Section~\ref{sec:CombiningTuckerChebyshev},
\revision{Eq.~\eqref{eq:FunctionalTuckerFormat}.}
Using Chebfun~\cite{Battles04}, the columns in $U_f,V_f,W_f$ are transformed into Chebyshev interpolants.
Following Chebfun3, we perform an accuracy verification for $\hat{f}$ by comparing its evaluations at Halton points to the original $f$. If the difference of the evaluations is too large, we restart the whole algorithm \revision{up to ten times} using a finer coarse grid. Additionally, we modify the ranks such that if $r_{{\sRevision{\ell}}_1} \leq 2$ we set $r_{{\sRevision{\ell}}_2} = \max(6,2r_{{\sRevision{\ell}}_2})$ and $r_{{\sRevision{\ell}}_1} = 3$ for ${\sRevision{\ell}}_1 \neq {\sRevision{\ell}}_2$, \revision{and after the forth restart we set $r_{\sRevision{\ell}}=2r_{\sRevision{\ell}}$ for ${\sRevision{\ell}}=1,2,3$.} This ensures that the multilinear ranks can grow in Phase~1. The overall Chebfun3F algorithm is formalized in Algorithm~\ref{alg:Chebfun3F}.

\begin{algorithm}
\caption{Chebfun3F}\label{alg:Chebfun3F}
\begin{algorithmic}[1]
\algdef{SE}[SUBALG]{Indent}{EndIndent}{}{\algorithmicend\ }%
\algtext*{Indent}
\algtext*{EndIndent}
\State \textbf{Input:} \revision{A function $f(x,y,z)$, a procedure $\T_f(i,j,k) = f(x_i^{(1)},x_j^{(2)},x_k^{(3)})$ that evaluates the values of $f$ on a $n_1^{(f)}\times n_2^{(f)}\times n_3^{(f)}$ Chebyshev grid on demand, a stopping tolerance $\eps>0$.}
\State \textbf{Initialization:} $(n_1^{(c)},n_c^{(c)},n_3^{(c)}) = (17,17,17)$, $(r_1,r_2,r_3) = (6,6,6)$
\State \textbf{Phase~1:}
\Indent
 \State apply Algorithm~\ref{alg:Phase1} to compute the factor matrices $U_c,V_c,W_c$ and to update $(n_1^{(c)}, n_2^{(c)}, n_3^{(c)})$, $(r_1,r_2,r_3)$
 \EndIndent
\State \textbf{Phase~2:}
\Indent
\State $(n_1^{(f)}, n_2^{(f)}, n_3^{(f)}) = (n_1^{(c)}, n_2^{(c)}, n_3^{(c)})$, \quad $U_f=U_c$, $V_f=V_c$, $W_f=W_c$
\While \text{Chebfun} heuristic in~\cite{Aurentz17} to decide if $U_f$ contains a sufficient number of entries is not satisfied
\State $n_1^{(f)} = 2n_1^{(f)}-1$, \quad refine $U_f$ such that its columns have length $n_1^{(f)}$
\EndWhile
\State proceed analogously to obtain $V_f,W_f$
\EndIndent
\State \textbf{Phase~3:}
\Indent
\State $[Q_{U}, \revision{\sim}] = \text{qr}({U}_f)$, \quad ${I} = \text{DEIM}(Q_{U})$, \quad $U  = Q_{U} \cdot Q_{U}({I},:)^{-1}$
\State $[Q_{V}, \revision{\sim}] = \text{qr}({V}_f)$, \quad ${J} = \text{DEIM}(Q_{V})$, \quad $V  = Q_{V} \cdot Q_{V}({J},:)^{-1}$
\State $[Q_{W}, \revision{\sim}] = \text{qr}({W}_f)$, \quad ${K} = \text{DEIM}(Q_{W})$, \quad $W  = Q_{W} \cdot Q_{W}({K},:)^{-1}$
\State compute the Chebyshev interpolants $\sRevision{u_i(x),v_j(y),w_k(z)}$ \sRevision{based on $U,V,W$} (see Section~\ref{sec:CombiningTuckerChebyshev})
\State $\mathcal{C} = T_f\bra{{I},{J},{K}}$,\quad $\hat{f}(x,y,z) \sRevision{ = \sum_{i=1}^{r_1} \sum_{j=1}^{r_2} \sum_{k=1}^{r_3} \mathcal{C}_{ijk} u_i(x) v_j(y) w_k(z)}$
\If $\abs{f(x,y,z)-\revision{\hat{f}(x,y,z)}} > 10\eps$ at Halton points $(x,y,z)\in [-1,1]^3$
\State  modify the ranks $r_{\sRevision{\ell}}$ if they are too small
\State restart from Phase~1 with $n_{{\sRevision{\ell}}}^{(c)} =\big\lfloor\sqrt{2}^{\floor{2\log_2(n_{{\sRevision{\ell}}}^{(c)}) + 1}}\big\rfloor + 1$
\EndIf
\EndIndent
\State \textbf{Output:} approximation $\hat{f}(x,y,z) \sRevision{ = \sum_{i=1}^{r_1} \sum_{j=1}^{r_2} \sum_{k=1}^{r_3} \mathcal{C}_{ijk} u_i(x) v_j(y) w_k(z)}$
\end{algorithmic}
\end{algorithm}  

\textit{Remark.} 
In Chebfun3 upper bounds for the multilinear rank, polynomial degree and grid sizes are prescribed.
The tolerances in the accuracy verification and in the ACA are initially set close to machine precision or provided by the user.
Tolerance issues are avoided by relaxing these tolerances adaptively based on the computed function evaluations.
In Chebfun3F, we handle these technicalities in the same manner.

\subsubsection{Existence of a Quasi-Optimal Chebfun3F Approximation}\label{sec:Chebfun3Ferror}

Due to the many heuristic ingredients in the Chebfun3F algorithm, it is difficult to analyze the convergence of the whole algorithm. Instead, we discuss the existence and error analysis of a specific Chebfun3F reconstruction. Lemma~\ref{lem:ErrorBoundF-FHat} shows how we can bound the approximation error depending on $\norm{\T-\hat{\T}}_\infty$. Theorem~\ref{thm:bestApproxError} \revision{}provides a bound for this error \revision{for a tensor approximation of the format~\eqref{eq:CF3FThat}} with specifically chosen fibers and index sets. The best approximation in the format will be at least as good. Although Chebfun3F is not guaranteed to return \revision{these specific fibers and index sets}, it is hoped that its error is not too far away.

\begin{theorem} \label{thm:bestApproxError}
Consider $\T \in \R^{n_1\times n_2 \times n_3}$ of multilinear rank at least $(r_1,r_2,r_3)$.
Let $Q_U,Q_V,Q_W$ denote orthonormal bases of $r_1,r_2,r_3$ selected mode-$1,2,3$ fibers of $\T$ respectively.
Given index sets $I,J,K$ we consider a Tucker decomposition of the form
\begin{equation*}
    \hat{\T} =  \T \times_1 Q_U (\Phi_I^T Q_U)^{-1} \Phi_I^T \times_2 Q_V (\Phi_J^T Q_V)^{-1} \Phi_J^T \times_3 Q_W (\Phi_K^T Q_W)^{-1} \Phi_K^T,
\end{equation*}
where $\Phi_I^T U = U(I,:)$. There exists a choice of fibers and indices such that 
\begin{align*}
    \norm{\T-\hat{\T}}_{\infty} \leq& \left(  \sqrt{\sRevision{q(r_1,n_1)\cdot}(r_1+1)} + \sqrt{\sRevision{q(r_1,n_1)\cdot q(r_2,n_2)\cdot}(r_2+1)} \right. \\ 
&+ \left.  \sqrt{\sRevision{q(r_1,n_1)\cdot q(r_2,n_2)\cdot q(r_3,n_3)\cdot} (r_3+1) } \right)  \norm{\T-\hat{\T}_{\mathsf{best}}}_F, 
\end{align*}
where $\norm{\cdot}_F$ denotes the Frobenius norm,  \sRevision{$q(r,n) = \sqrt{1 + r(n-r) }$,}  and $\hat{\T}_{\mathsf{best}}$ is the best Tucker approximation of $\revision{\T}$ with multilinear rank at most $(r_1,r_2,r_3)$.
\end{theorem} 

\begin{proof}
Using Frobenius norm properties and Lemma~\ref{lem:DEIMerror}, we obtain
\begin{align}
\norm{\T-\hat{\T}}_\infty \leq& 
\norm{\T-\hat{\T}}_F 
\leq \norm{(\Phi_I^T Q_U)^{-1}}_2 \norm{(I-Q_UQ_U^T) \T^{\set{1}}}_F  \nonumber \\
&+ \norm{(\Phi_I^T Q_U)^{-1}}_2 \norm{(\Phi_J^T Q_V)^{-1}}_2 \norm{(I-Q_VQ_V^T)\T^{\set{2}}}_F \nonumber \\
&+ \norm{(\Phi_I^T Q_U)^{-1}}_2  \norm{(\Phi_J^T Q_V)^{-1}}_2  \norm{(\Phi_K^T Q_W)^{-1}}_2  \norm{(I-Q_WQ_W^T)\T^{\set{3}}}_F.
\label{eq:T-hatTinTheorem2}
\end{align} 
From~\cite[Lemma 2.1]{Goreinov97} it follows that there exists an index set $I$ such that 
\begin{equation}\label{eq:bound1Theorem2}
\norm{(\Phi_I^T \revision{Q_U})^{-1}}_2 \leq \sqrt{1 + r_1(n_1-r_1)}.
\end{equation}
From~\cite[Theorem 8]{Deshpande10} \revision{with the role of rows and columns interchanged} it follows that we can select mode-$1$ fibers \revision{$U$} of $\T$ such that 
\begin{equation} \label{eq:bound2Theorem2}
\norm{(I-Q_UQ_U^T)\T^{\set{1}}}_F = \norm{(I-U(U^TU)^{-1}U^T)\T^{\set{1}}}_F  \leq  \sqrt{r_1+1} \norm{\T-\hat{\T}_{\mathsf{best}}}_F.
\end{equation} 
Analogous bounds hold for $\norm{(\Phi_J^T Q_V)^{-1}}_{2}$, $\norm{(\Phi_K^T Q_W)^{-1}}_{2}$, $\norm{(I-Q_VQ_V^T)\T^{\set{2}}}_F$ and $\norm{(I-Q_WQ_W^T)\T^{\set{3}}}_F$.
Applying the bounds~\eqref{eq:bound1Theorem2} and~\eqref{eq:bound2Theorem2} to the factors in~\eqref{eq:T-hatTinTheorem2} yields the claimed result.
\end{proof}

\textit{Remark.} If one uses orthogonal instead of oblique projections in Phase~3, Corollary~6 in~\cite{Cortinovis19} yields a bound similar to Theorem~\ref{thm:bestApproxError}. 
\revision{Whilst the index sets obtained from DEIM~\cite{Chaturantabut10} yield small errors in practice, their theoretical upper bounds for $\norm{(\Phi_I^T {Q_U})^{-1}}_2$ grow exponentially in $r$. In contrast, the strong rank-revealing QR decomposition~\cite{Gu96} yields an index set for which a bound similar to Inequality~\eqref{eq:bound1Theorem2} is known~\cite[Lemma 2.1]{Drmac18}.}

\revision{\textit{Remark.}  Note that the bound in Theorem~\ref{thm:bestApproxError} is the worst case bound for the optimal choice of index sets. However, this does not present the whole picture as even suboptimal index sets might yield a much better approximation in practice. To quantify the quality of the approximation in practice, we computed Chebfun3F approximations for the functions $f(x,y,z) = \log(1+x^2+y^2+z^2)$, $f(x,y,z) = 1/(1+x^2+y^2+z^2)$, $f(x,y,z) = \exp(xyz)$ and compared $\norm{\T_f-\hat{\T}}_\infty$ and $\norm{\T_f-\T_{\textsc{HOSVD}}}_\infty$, where $\hat{\T} = \mathcal{C}\times_1 U \times_2 V \times_3 W$ and 
$\T_{\textsc{HOSVD}}$ denotes the truncated HOSVD of $\T_f$ with multilinear ranks equal to those of $\hat{\T}$. We observed that both errors differ by at most a factor of $2$. Even though the truncated HOSVD does not focus on $\norm{\cdot}_\infty$, it still can serve as a good proxy. Hence by Lemma~\ref{lem:ErrorBoundF-FHat}, the error in Chebfun3F is comparable to the error obtained from the truncated HOSVD.}

%\revision{\textit{Remark.}  The total error in the function approximation can be bounded using Lemma~\ref{lem:ErrorBoundF-FHat}.
% The stopping tolerance $\eps$ in Alg~.\ref{alg:Chebfun3F} can be seen as an error threshold in the discrete tensor.
% The Chebyshev interpolation error and the Lebesgue constants are taken into account in Line~16 of Alg.~\ref{alg:Chebfun3F}}

\subsubsection{Comparison of Theoretical Cost}
Assume $f$ can be approximated accurately in Tucker format~\eqref{eq:FunctionalTuckerFormat} with multilinear rank $(r,r,r)$ and polynomial degrees $(n,n,n)$, $n\geq r$.
In a highly idealized setting Chebfun3 and Chebfun3F refine the coarse grid in Phase~1 until $n^{(c)}_{{\sRevision{\ell}}} > (2\sqrt{2})r$ and identify fibers on this coarse grid. These fibers are refined until $n^{(f)}_{{\sRevision{\ell}}} \geq n$ and lead to Tucker approximations which pass the accuracy check in Phase~3.
Under these circumstances, both Chebfun3 and Chebfun3F use $\mathcal{O}(r^3)$ function evaluations in Phase~1 (see Section 2.2 in~\cite{Hashemi17}). In total, Chebfun3 requires $\mathcal{O}(nr^2)$ function evaluations~\cite[Proposition 2.1]{Hashemi17}), whereas Chebfun3F only requires $\mathcal{O}(r^3+nr)$, since fewer fibers are refined in Phase~2.
We want to emphasize that, in general, it is not guaranteed that this $n^{(c)}_{{\sRevision{\ell}}}$ suffices to identify fibers leading to an accurate approximation.
\revision{We summarize the breakdown of anticipated costs in Table~\ref{tab:cost-th}.}

\begin{table}[htb]
\centering
\begin{tabular}{c|cc}
          &  Chebfun3                    & Chebfun3F \\ \hline
Phase 1   &  $\mathcal{O}(r^3)$          & $\mathcal{O}(r^3)$       \\
Phase 2   &  $\mathcal{O}(nr^2)$         & $\mathcal{O}(nr)$   \\
Phase 3   &  0                           & $\mathcal{O}(r^3)$       \\ \hline
Total     &  $\mathcal{O}(r^3 + nr^2)$   & $\mathcal{O}(r^3 + nr)$ \\
\end{tabular}
\caption{\revision{Anticipated numbers of functions evaluations in Chebfun3 and Chebfun3F.}}
\label{tab:cost-th}
\end{table}

%% file: 5Results.tex
\section{Numerical Results}\label{sec:NumericalResults}
In this section, we present numerical experiments\footnote{The \revision{MATLAB} code to reproduce these results is available from \texttt{https://github.com/cstroessner/Chebfun3F}.} to compare Chebfun3F and Chebfun3. The main focus lies on the number of function evaluations required to compute the approximation in Tucker format~\eqref{eq:FunctionalTuckerFormat}. Unless mentioned otherwise the tolerance for the ACA and the accuracy check are initially set close to machine precision.

\subsection{Chebfun3 vs. Chebfun3F}
In Section~\ref{sec:Chebfun3Redundant} we illustrated that the function
\begin{equation*}
    f(x,y,z) = \frac{1}{1+25\sqrt{x^2+y^2+z^2}},
\end{equation*}
leads to a lot of redundant fibers in Chebfun3. For this function, for both Chebfun3 and Chebfun3F, the accuracy check at the end of Phase~3 fails and the computation is restarted on a finer coarse grid, which leads to an approximation that passes the test.
Overall $903\,364$ function evaluations are required by Chebfun3, $421\,041$ of them are used before the restart. In comparison, Chebfun3F only needs $222\,546$ function evaluation in total and $100\,496$ before the restart. In the final accuracy check, the estimated error for Chebfun3 is $7.8\cdot10^{-13}$ and $3.6\cdot10^{-13}$ for Chebfun3F, i.e. both approximations achieve around the same accuracy. 

In Figure \ref{fig:EvalsPerPhase}, the function evaluations per phase are juxtaposed.
The figure shows, that Chebfun3 requires a large number of function evaluations in Phase~2. This is caused by the redundant fibers. In Phase~1, Chebfun3F requires fewer function evaluations, since $\T_c$ is not evaluated completely. Whilst Chebfun3 only requires function evaluations in Phase~3 for the accuracy check, Chebfun3F additionally needs to compute the core tensor, which only leads to a small number of function evaluations compared to the other phases.

\begin{figure}[ht]
\centering
\subfloat[Before Restarting]{{\includegraphics[width=0.4\textwidth]{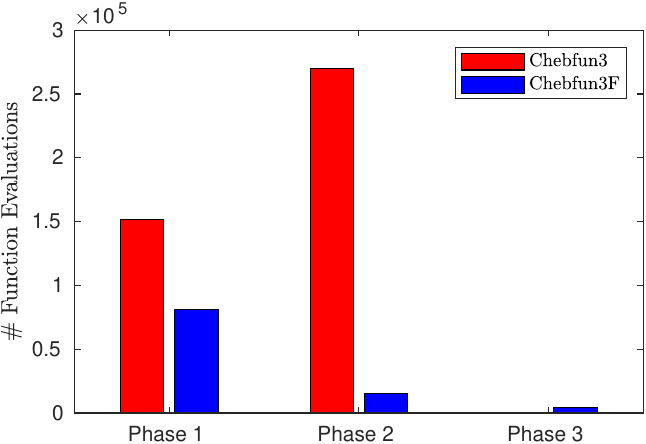}}}%
\qquad
\subfloat[After Restarting]{{\includegraphics[width=0.4\textwidth]{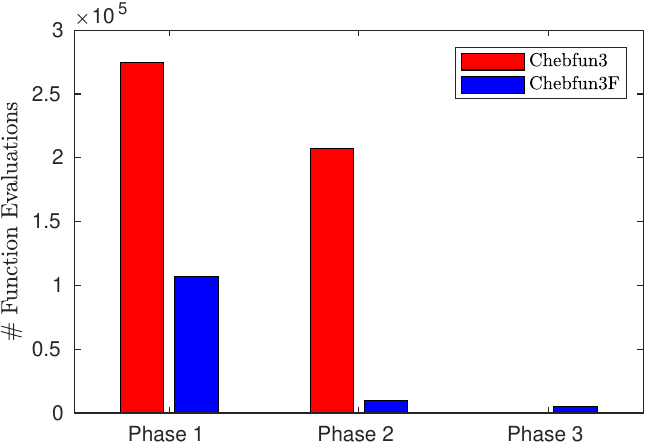}}}%
\caption{Comparison of the function evaluations used by Chebfun3 and Chebfun3F to approximate $f(x,y,z) = (1+25\sqrt{x^2+y^2+z^2})^{-1}$. Both algorithms are restarted due to failing the accuracy check once. Evaluations \revision{}before the restart are depicted in (a), evaluations after the restart in (b). The evaluations are subdivided into phases corresponding to the phases in Section~\ref{sec:Chebfun3} and Section~\ref{sec:Chebfun3F} respectively.}
\label{fig:EvalsPerPhase}
\end{figure}

\revision{\textit{Remark.} The number of function evaluations required by Chebfun3F depends on the random initialization of the index sets $\tilde{I},\tilde{J},\tilde{K}$ in Algorithm~\ref{alg:Phase1}. We computed the Chebfun3F approximation of $f$ for $1\,000$ different random initializations and observed numbers of function evaluations ranging from $213\,391$ to $226\,073$ with mean $221\,802.6$ and variance $4.96\cdot10^6$. Similarly mild fluctuations have been observed for all other functions tested.} 

In Figure~\ref{fig:ExhaustiveExamples}, the required function evaluations are depicted for four different functions. \revision{The corresponding computing times are depicted in Table~\ref{tab:runtime}.}  Again both algorithms lead to approximations of similar accuracy and Chebfun3F requires fewer function evaluations than Chebfun3.
For 
\begin{equation}\label{eq:F5a}
    f_1(x,y,z) = \exp(-\sqrt{(x-1)^2+(y-1)^2+ (z-1)^2})
\end{equation}
Chebfun3 requires the refinement of a huge number of redundant fibers in Phase~2.
The evaluations for 
\begin{equation}\label{eq:F5b} %label is necessary but not called explicitly
f_2(x,y,z) = \revision{[}\cosh\revision{(}3\revision{(}x+y+z\revision{))]^{-2}}
\end{equation}
differ most in Phase~1, where Chebfun3F benefits from not evaluating $\T_c$ completely. No additional refinement is required in Phase~2.
For the function 
\begin{equation}\label{eq:F5c}
f_3(x,y,z) = \frac{10^5}{1+10^5(x^2+y^2+z^2)}.
\end{equation}
Chebfun3F reduces the number of required function evaluations by more than $98\%$ to $1\,603\,693$ from $109\,269\,332$ required by Chebfun3. 

\begin{figure}[!ht]
\centering
\subfloat[]{{\includegraphics[width=0.4\textwidth, height=0.27\textwidth]{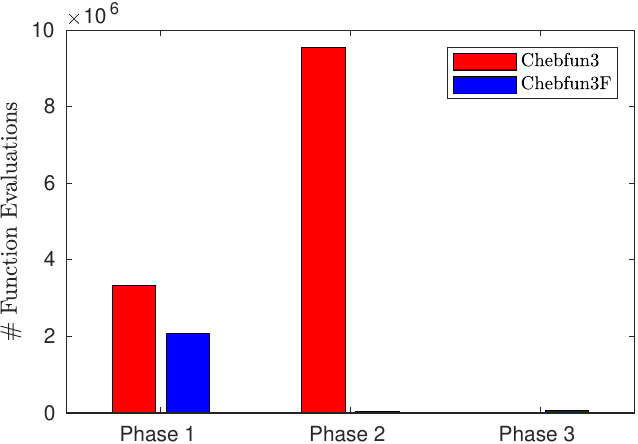}}}%
\qquad
\subfloat[]{{\includegraphics[width=0.4\textwidth, height=0.27\textwidth]{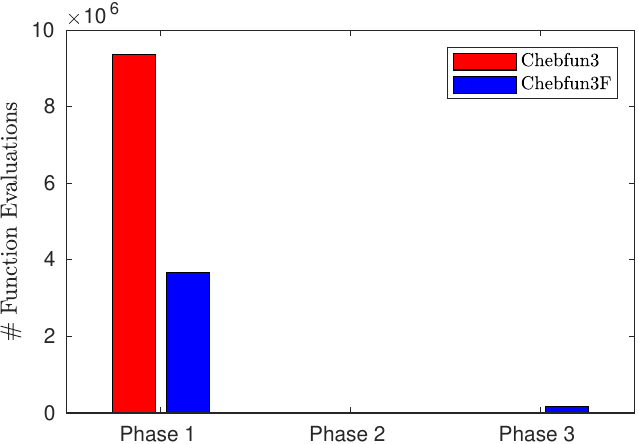}}}%
\qquad
\vspace{0.1cm}
\subfloat[]{{\includegraphics[width=0.4\textwidth, height=0.27\textwidth]{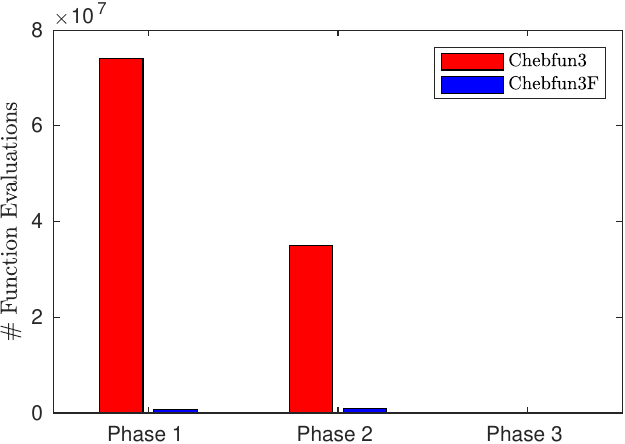}}}%
\qquad
\subfloat[]{{\includegraphics[width=0.4\textwidth, height=0.27\textwidth]{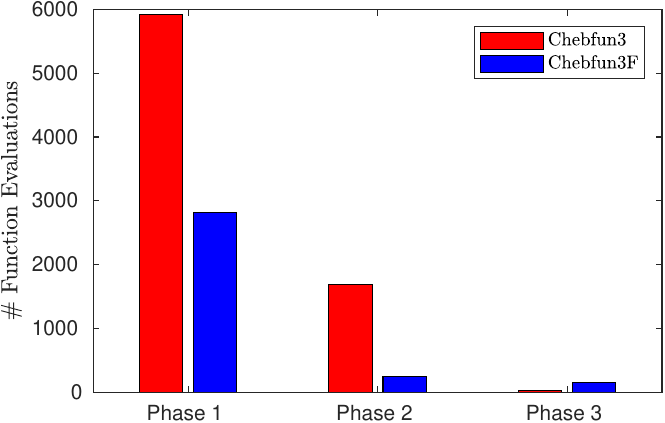}}}%
\caption{Total number of function evaluations per Phase in Chebfun3 and Chebfun3F for the functions: (a)~$f_1$ (b)~$f_2$, (c)~$f_3$ as defined in~\eqref{eq:F5a}-\eqref{eq:F5c} (d) the parametric PDE model~\eqref{eq:ModelPDE1}-\eqref{eq:ModelPDE2} with prescribed tolerance $10^{-9}$.}
\label{fig:ExhaustiveExamples}
\end{figure}

\begin{table}[!ht]
\centering
 \begin{tabular}{|c |c c c c|} 
 \hline
& $f_1$ & $f_2$ & $f_3$ & PDE \\ 
 \hline
 Chebfun3 & $3.44$  & $3.20$ & $15.80$ & $408.62$ \\ 
 \hline
 Chebfun3F & $0.30$ & $0.77$ & $0.56$ & $100.19$ \\
 \hline
\end{tabular}
\caption{\revision{Computing times in MATLAB in seconds for Chebfun3 and Chebfun3F approximations of $f_1,f_2,f_3$ as defined in~\eqref{eq:F5a}-\eqref{eq:F5c} and of the parametric PDE model~\eqref{eq:ModelPDE1}-\eqref{eq:ModelPDE2} with prescribed tolerance $10^{-9}$.}}
\label{tab:runtime}
\end{table}

In certain cases Chebfun3 outperforms Chebfun3F. 
This can happen in degenerated situations. 
For instance, the function $f(x,y,z) = \tanh(5(x+z))\exp(y)$ requires \sRevision{a} Tucker decomposition with rank $r = [71,1,71]$.
In this case, Chebfun3F heavily relies on the heuristic to increase $(r_1,r_2,r_3)$ when restarting and requires $1\,641\,712$ function evaluations compared to $1\,128\,061$ in Chebfun3.
Other functions that are difficult to approximate with Chebfun3F include \revision{(numerically)} locally supported functions, \revision{such as a trivariate normal distribution with very small entries in the covariance matrix,} for which identifying \revision{non-zero} fibers in Phase~1 without fully evaluating $\T_c$ might be difficult. 

\subsubsection*{Application: Uncertainty Quantification}
%UQ application
Algorithms in uncertainty quantification, such as the Metropolis-Hastings method, often require the repeated evaluations of a parameter depended quantity of interest \cite{Stuart10}.
In many applications, the evaluation of this quantity requires the solution of a PDE depending on the parameters.
To speed up computations, the mapping from the parameters to the quantity of interest is often replaced by a surrogate model \cite{Xiu17}.
In the context of models with three parameters (or after a dimension reduction to three parameters~\cite{Constantine16}), Chebfun3/Chebfun3F could be a suitable surrogate.

We consider the parametric elliptic PDE model problem on $\Omega = [-1,1]^2$
\begin{align}
\nabla \revision{\cdot((}(p_1+2)f_1(x,y) + (p_2+2) f_2(x,y) + (p_3+2) f_3(x,y)\revision{)} \nabla u(x,y)\revision{)} &= 1 && (x,y) \in \Omega, \label{eq:ModelPDE1}\\
u(x,y) &= 0 && \in \delta\Omega, \label{eq:ModelPDE2}
\end{align}
with parameters $(p_1,p_2,p_3)  \in [-1,1]^3$ and functions $f_1(x,y) = \cos(x) + \sin(y) + 2$, $f_2(x,y) = \sin(x) + \cos(y) + 2$, and $f_3(x,y) = \cos(x^2+y^2) + 2$.
The quantity of interest is defined as point evaluation $u(0.5,0.5)$.
With Chebfun3F, we only need $3,217$ PDE solves to compute an approximation with prescribed accuracy $10 ^{-9}$, whereas Chebfun3 requires $7,626$ as depicted in Figure~\ref{fig:ExhaustiveExamples}(d).

\subsection{Comparison to Sparse Grids}
Lastly, we study how efficient Chebfun3 approximations are compared to sparse grids \cite{Bungartz04}.
Sparse grids are a method to interpolate functions by projecting them onto a particular space.
This space is obtained by selecting the most beneficial elements from a hierarchical basis under the assumption that the function has bounded mixed second derivatives.
Interpolation based on sparse grids performs particularly well when the norms of the mixed second derivatives of the function are small.
In the following, we use dimension adaptive sparse grids based on a Chebyshev-Gauss-Lobatto grid with polynomial basis functions from the Sparse Grid Interpolation Toolbox~\cite{Klimke08}.

\begin{figure}[ht]
\centering
\subfloat[]{{\includegraphics[width=0.4\textwidth]{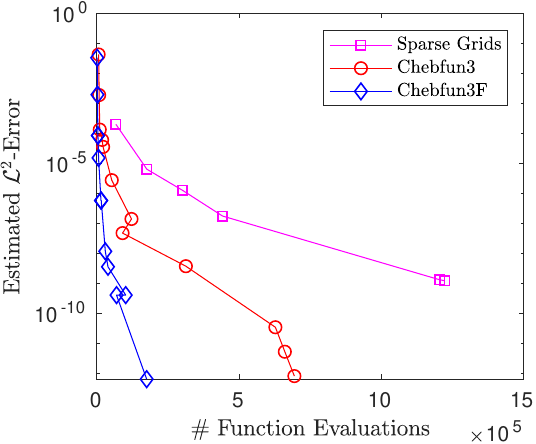}}}%
\qquad
\subfloat[]{{\includegraphics[width=0.4\textwidth]{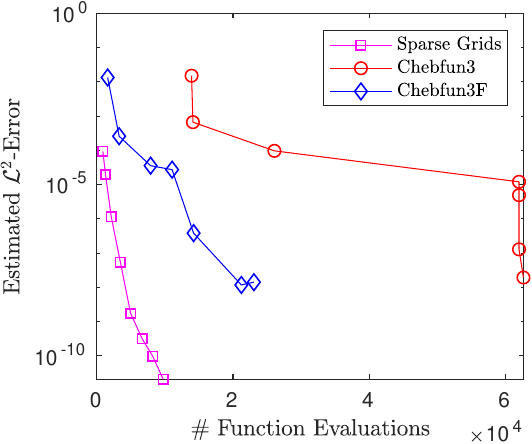}}}%
\qquad
\vspace{0.1cm}
\subfloat[]{{\includegraphics[width=0.4\textwidth]{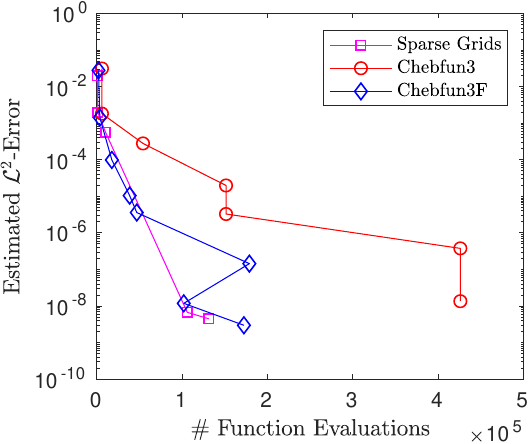}}}
\caption{Comparison of the number of function evaluations required to compute a Chebfun3, Chebfun3F and sparse grid approximation for the functions: (a)~$f(x,y,z) = (1+25\sqrt{x^2+y^2+z^2})^{-1}$, (b)~sum of $10$ Gaussians, (c)~$f_4$ \revision{as defined in}~\eqref{eq:F6c}. The algorithms are initialized with varying tolerances. Their $\mathcal{L}^2$ error is estimated at $1\,000$ sample points.}
\label{fig:SGcomparision}
\end{figure}

We compare how the approximation error decays compared to the number of function evaluations.
Therefore, we prescribe varying tolerances to the algorithms.
In Figure \ref{fig:SGcomparision} the error decay is plotted for sparse grids, Chebfun3 and Chebfun3F.
In (a), the function already studied in section \ref{sec:Chebfun3Redundant} is depicted.
We observe that both Chebfun3F and Chebfun3 require fewer function evaluations than sparse grids to achieve the same accuracy.
The sparse grids perform poorly, since the function is smooth, but the norms of the second mixed derivatives are rather large. 
In contrast, the sum of $10$ Gaussians depicted in (b) is well suited for sparse grids.
In this case sparse grids require fewer function evaluations than Chebfun3F and Chebfun3 for the same accuracy.
In (c) the Chebfun3F and sparse grids perform about equally well for
\begin{equation} \label{eq:F6c}
    f_4(x,y,z) = \log(x+yz+\exp(xyz)+\cos(\sin(\exp(xyz)))). 
\end{equation}
For an arbitrary, black-box function it is not clear a priory whether a sparse grid interpolation or a Tucker decomposition~\eqref{eq:FunctionalTuckerFormat} is the more efficient type of approximation\revision{. However, when the Tucker decomposition is the better approximation format, we can expect that Chebfun3F requires fewer function evaluations compared to Chebfun3.}

%% file: 6Conclusions.tex
\section{Conclusions}
Trivariate functions defined on tensor product domains can be approximated efficiently by combining tensorized Chebyshev interpolation and a low-rank Tucker approximation of the evaluation tensor. In this paper, we presented Chebfun3F to compute such approximations. Our numerical experiments show that Chebfun3F requires fewer function evaluations to compute such an approximation of the same accuracy compared to Chebfun3. Future work could cover how operations can be computed directly on the level of Tucker decompositions. For instance, multiplication can be treated directly on the tensor level~\cite{Kressner17}, whereas the Chebfun3 package relies on constructing a new approximation from point evaluations. We suspect that other operations can be computed in a similar manner.

Finally, let us remark that the extension of the presented algorithms and results to functions depending on more than three variables is trivial. However, the Tucker format is not well suited for the high-order tensors arising from the evaluation of a function in many variables. Other formats, such as the TT format, are better suited for this purpose and will require different construction algorithms.